\documentclass[a4paper,12pt,reqno]{article}
\usepackage{amsmath}
\usepackage{amsfonts}
\usepackage{amssymb}
\usepackage{graphicx}
\newtheorem{theorem}{Theorem}

\newtheorem{argument}[theorem]{Argument}
\newtheorem{corollary}[theorem]{Corollary}

\newtheorem{definition}[theorem]{Definition}

\newtheorem{lemma}[theorem]{Lemma}

\newtheorem{remark}[theorem]{Remark}

\newenvironment{proof}[1][Proof]{\textbf{#1.} }{\ \rule{0.5em}{0.5em}}

\begin{document}
\title{SOME ELEMENTS OF A POSSIBLE DEMONSTRATION OF THE COLLATZ CONJECTURE}
\author{Denis Mart\'{i}nez T\'{a}panes, Jose E. Mart\'{i}nez Serra}
\date{\small Mathematical Department, Universidad Central de Las Villas (Marta Abreu University), Cuba, Santa Clara City (zip code 50100),  November 11, 2015}
\maketitle
%.................................................................................................................................................................................................................................................................................................................................................................................................................................................................................................................................................................
\begin{abstract}In this paper we show the following facts: The probability of increasing $ A_{k}=P(T^{k} (x_{0})>T^{k-1} (x_{0})) $, and the probability of decreasing $B_{k}=P(T^{k} (x_{0})<T^{k-1} (x_{0}))$ in step $ k $ of a Collatz  procedure initiated in $ x_ {0} \in \mathbb{N} $ arbitrary, they are equal for all values of $ k $. This influences on the law that generates the numbers of a Collatz sequence so that it is forced to decrease until the unit. It is also shown that in the Collatz conjecture is false for every problem $an +b$ such that $a\geq 5\geq b+2$, and its probabilistic character can not be ignored if you want to get to the definitive solution, among other interesting arguments.
\end{abstract}
%.................................................................................................................................................................................................................................................................................................................................................................................................................................................................................................................................................................
\begin{quotation}
\textbf{keywords}: Conjecture, Collatz, Procedures, Equivalence, Probability.
\end{quotation}
%.................................................................................................................................................................................................................................................................................................................................................................................................................................................................................................................................................................
\section{Introduction}
In $1937$, the German mathematician Lothar Collatz states a hypothesis that has remained intact until today. His affirmation today takes different names: Collatz conjecture, $3n + 1$ problem, Syracuse problem, etc.\\Collatz conjecture predicts that if any natural number is subjected to certain operations, given by a very specific procedure, a sequence of numbers that tends to a unique cycle always is produced. These operations are in general (Starting at an arbitrary natural number):
\begin{enumerate}
\item If the initial number is even, then is divided by $2$ successively until an odd number and turn to the step $2$.
\item If the number is odd multiply by $3$ and the result is added $1$, therefore you arrives to an even number and turn to step $1$.
\end{enumerate}
As an example may be the case of the number $5$:
\[\left(5\cdot3+1=16\right)\longrightarrow\left(\frac{16}{2}=8\right)\longrightarrow\left(\frac{ 8}{2}=4\right)\longrightarrow\left(\frac{4}{2} =2\right)\longrightarrow\]
\[\longrightarrow\left(\frac{2}{2}=1\right)\longrightarrow\left(1\cdot3+1=4\right)\longrightarrow\left(\frac{4}{2}=2\right)\longrightarrow\left(\frac{2}{2}=1\right)\]
Case of number  $12$:
\[\left(\frac{12}{2}=6\right)\longrightarrow\left(\frac{6}{2}=3\right)\longrightarrow\left(3 \cdot 3+1=10\right)\longrightarrow\left(\frac{10}{2}=5\right)\ldots\]
The process then continues  according to case of the number $5$, already seen.\\Several investigations are carried out since the second half of the last century to the present trying to obtain a general expression of the existence of the cycle in $3x + 1$ problem and others try to verify the conjecture of the maximum quantity of numbers that meet the conjecture, using computational methods. Many of these studies are summarized in Lagarias \cite{1}$ (2011) $ and Lagarias \cite{2}$(2012)$. However, no case it has the solution, which seems to be still far off, judging by the results, although interesting conjectures and relationships are established with many other branches of mathematics.\\Our approach does not it consider superior to any other, simply we try, in the most direct way possible, resolve the problem by addressing the same process that occurs as generation of numbers starting in an initial number. We think the only way to solve the problem is to discover the laws governing the succession generation, which must be achieved by analyzing the general mathematical expression of this process given in Lemma \ref{lemma8}. Said expression contains variables that depend on how much the number generation process runs before arriving in a cycle. But we found, too, that the variables have an unquestionable dependence of certain probabilities associated with the sign of the difference $x_{i}-x_{i+1}$ for two successive numbers ($x_{i}$ y $x_{i+1}$) that they are generated by what we have called, ``Collatz procedure''.\\It should be clear that the probability that in this work is handled is not determined on a empirical way, by the frequency with which certain events occur, on the contrary it is a classic probability. In Theorem \ref{theorem10} it is stated that exist exactly the same amount, within certain set $\Gamma_ {M}$, of natural numbers of an certain class whose elements have behavior $ x $ and another, whose elements have a behavior $ y $ different. So if it is removed of $\Gamma_ {M} $ one random element, the probability will be $\frac{1}{2}$ for $ x $ , and $\frac{1}{2}$ for $ y $ respectively. This excludes the possibility of ``strange behaviors '' outside these two exclusionary behaviors.\\The theorem \ref{theorem9}, meanwhile, shows that with these probabilities it is ensured that the numbers generated during the procedure Collatz, initiated in any natural number, will be ``forced'' to decrease to a smaller value than any predetermined natural number. This paper says nothing about how this decrease is done and it is clear that for each particular number, the succession should also be very particular too; what we do clearly is to show that, at some point, must be decreasing up to the unit.\\ The Lagarias work, already mentioned, lists the latest results and is an important guide to deal with the problem. However none of the registered work makes a demonstration itself but interesting studies are presented. Unfortunately we did not find in this work an approach that go in the direction that we raised here, and most of them presented the results in such a general way, that is lost a lot of information, as is the case of Bohm and Sontacchi \cite{3}$(1978) $ where it is stated that the veracity of the Collatz conjecture is equivalent to say that every positive integer $ n $ can be written as,
\[n=\frac{1}{3^{m}}\left(2^{v_{m}}-\sum^{m-1}_{k=0}3^{m-k-1}2^{v_{k}}\right),\]
where $0\leq v_0 < v_1 < \ldots v_m$ are integers; but from our point of view the question is, precisely, in these integers, and in the way they have as sums of the exponents of $2$ listed in the corresponding equations to Collatz procedure (\ref{eqn10}) which will be seen below. In fact, in essence, expression that is shown in lemma \ref{lemma8} is the same raised by these authors, assuming that the sequence of numbers obtained, has come to the unit and whereas in addition $ n = x_0 $. However we prefer to perform a ourselves demonstration and show the shape of the $\ldots v_m\ldots$. Also showing that $ v_m=m\bar{k}$, which is fundamental here for statistical considerations.\\Although the authors of \cite{6} are approaching to the objective using a probabilistic approach, they do not combine the result related to the probability with the law (The mathematical expression given in the lemma \ref{lemma8}) of Collatz procedure, but we, on the contrary, do this in the theorem \ref{theorem9}; furthermore, they consider the probability (value) as hypotheses, while the classical probability that is handled in our work, is fully demonstrated in Theorem \ref{theorem10}. So in our case we are dealing with a full demonstration, but not of a type approach `` If $ A $ it is demonstrated  then $ B $ is true'' but rather ``$ B $ is true''.\\Meanwhile, in the section on findings, we show the necessary and indispensable statistical character of the $3n + 1$ problem, so any other approach undoubtedly would establish relations and interesting ideas, but not would solve the Collatz conjecture.\\ In addition, we point the very important fact, which is the application of the demonstration presented in this paper to the general problem , obtaining the result expected; i.e., \emph{The Collatz conjecture is false for the general problem  $an + b $}.\\
Now, we relate some arguments that have been raised, against this article so as to serve as a starting point for future revisions and to reaffirm the arguments of our work, respecting the opinions of others, but clarifying ideas could raise questions to other colleagues:
\begin {argument}
``\ldots the statement of the first main result (that is, Theorem \ref{theorem9}) is relatively clear.  Unfortunately, it's not possible to make sense of this statement mathematically.  In particular, the statement claims that two probabilities are equal.  But in these two probabilities, everything is fixed, so the probabilities involved are statements like $Pr(5 > 3) = Pr(5 < 3)$, which is clearly false."
\label{arg1}
\end {argument}
The proposition $ Pr (5> 3) = Pr (5 <3) $ can not be equivalent to no proposition, since the propositions are true or false, whereas this no sense, we can say false things, if that we are wrong, but not things completely meaningless; it not even makes sense the proposition $ Pr (5> 3) $ because $ 5> 3$ is a truth that is not to be considered from the point of view of the theory of probability.\\
In fact, we treat the probability that a particular event well lead to a well-defined outcome; ie the random extraction of a natural number of a finite set of natural numbers, implies that it chose a process Collatz also random (All natural has its own process Collatz), then we show that in step $ k $ of this process, randomly chosen, the probability that the number obtained is greater than the number obtained in step $ k-1 $ equals $ 1/2 $. Therefore, in step $ k $ of this process, randomly chosen, the probability that the number obtained is less than the number obtained in step $ k-1 $ equals also equal to $ 1/2 $.\\
Meanwhile the action of extracting, at random, of a natural number of a finite set, is not unusual and can to be always, in principle,  by simple enumeration of objects (for example balls) and placing them in an urn.

\begin {argument}
``The statement also does not appear to be salvageable, because the quantities that could vary must vary over the natural numbers, over which the authors seem to want a uniform probability distribution.  But no such distribution exists."
\label{arg2}
\end {argument}
Of course the natural numbers do not have uniform distribution, but we do not deal here with the distribution of natural numbers, but with the distribution of the results obtained in the experiment that involves extraction a natural number from a finite set of natural numbers, which is perfectly feasible in principle as discussed above. But, as already explained, the random extraction of a natural number is equivalent to the random extraction of the process of Collatz. Particularly if we assign to step from a smaller number to a larger (an increase) the number $1$, and step from a higher number to a lower (decreased) zero, we have obtained a random sequence of zeros and ones whose distribution will be normal to the extent that the amount of these elements is greater in accordance with the law of large numbers.

%.................................................................................................................................................................................................................................................................................................................................................................................................................................................................................................................................................................
\section{Definitions}
\begin{definition}General Collatz Procedure (\textbf{g.c.p.}). It is so called a sequence of numbers initiated in a number $ x_ {0} \in \mathbb {N} $, which follow the following scheme:
\[\left\{T^{0}(x_{0}),T^{1}(x_{0}),T^{2}(x_{0})\ldots,T^{n}(x_{0}),\ldots\right\},\]
where	
\[T^{0} (x_{0} )  =x_{0}, \]
and
\[T^{i} (x_{0})=\left\{\begin{array}{c}
\frac{T^{i-1} (x_{0})}{2}, \quad  T^{i-1} (x_{0})\in2\mathbb{N}\\\\
\frac{3T^{i-1}(x_{0})+1}{2}, \quad T^{i-1} (x_{0})\in2\mathbb{N}+1	 
\end{array}\right\}.\]
\label{definition3}
\end{definition}
\begin{definition}Odd Collatz Procedure (\textbf{o.c.p.}). It is so called a sequence of numbers initiated in a number $ x_ {0} \in 2\mathbb {N}+1$, which follow the following scheme:
\[\left\{T^{0}(x_{0}),T^{1}(x_{0}),T^{2}(x_{0})\ldots,T^{n}(x_{0}),\ldots\right\},\]
where
\[T^{0} (x_{0} )= x_{0}\in2\mathbb{N}+1 \]
\[T^{i}(x_{0} )= \frac{3 T^{i-1}(x _{0})+1}{2^{k_{i} }},\quad k_{i}\in\mathbb{N},\quad T^{i} (x_{0})\in2\mathbb{N}+1.\]
It is easy to verify that \textbf{o.c.p.} can be represented by a successive set of equations as follows:
\begin{equation}
\begin{array}[pos]{c}
3 x_{0}+1  = 2^{k_{1}} T^{1}(x_{0})\\
3 T^{1}(x_{0})+1=2^{k_{2}} T^{2}(x_{0})\\
\vdots\\
3 T^{n-1}(x_{0})+1=2^{k_{n}} T^{n}(x_{0})\\
\vdots
\end{array}
\label{eqn10}
\end{equation}
\label{definition4}
\end{definition}
\begin{definition}Collatz Step (\textbf{c.s.}). It is so called to the step taken from $T^{i} (x_{0})$ until $T^{i+1} (x_{0})$ in the general procedure given in the definition \ref{definition3}.
\item \textbf{Definition \ref{definition5}.1.}Decrement Collatz Step(\textbf{d.c.s.}). Occurs if you have $T^{i+1} (x_{0})<T^{i} (x_{0})$.  
\item \textbf{Definition \ref{definition5}.2.}Increase Collatz Step(\textbf{i.c.s.}). Occurs if you have $T^{i+1} (x_{0})>T^{i} (x_{0})$.
\label{definition5}  
\end{definition}
\begin{definition}Collatz Number(\textbf{c.n.}). That's any natural number such that  \textbf{\emph{g.c.p.}} started in it, tend to the cycle $[2,1]$.
\label{definition6}
\end{definition}
\textbf{Statement} \emph{Collatz Conjecture(\textbf{c.c.}).Given the above definition this conjecture states that all natural numbers are \textbf{c.n.}.}
%.................................................................................................................................................................................................................................................................................................................................................................................................................................................................................................................................................................
\section{Propositions}
\begin{lemma}Let $ 2^{k} m+i (0 \leq i< 2^k) $ be an equivalent class of additive group $ \mathbb{Z}_{2^k} $ module $ 2^k $, then the relationship $T^k (2^{k} m+i)=3^{p_{k}} m+T^{k} (i) $ is satisfied, being $p_{k} $ the amount of \textbf{i.c.s.} that occur until step $k$.
\label{lemma7}
\end{lemma}
\begin{proof}
In fact, let $ p_{r} $ be amount of \emph{\textbf{i.c.s.}} in $r$ steps, first we show that the following relation is satisfied
\[T^{r} (2^{k} m+i)=3^{p_r} 2^{k-r}m+T^{r} (i),\]
if $ r=1 $, then
\[ T(2^{k} m+i)=\left\{
\begin{array}[pos]{c}
2^{k-1} m+\frac{i}{2}, i \in 2\mathbb{N}\\\\
2^{k-1} 3 m+\frac{3 i+1}{2}, i \in 2\mathbb{N}+1
\end{array}
\right\}=\{2^{k-1} 3^{p_1} m+T(i)\}.\]
Obviously $ p_1$ is to be $0$ or $1$ respectively if $i$ is even or odd; it is clear that the relationship is true if $ r = 1$. Let $ r $ be an arbitrary amount of steps, if we assume that the relationship is true for $ r $, then for $ r + 1 $ will be:
\[T^{r+1} (2^{k} m+i)=\left\{
\begin{array}[pos]{c}
3^{p_r} 2^{k-(r+1)} m+\frac{T^r (i)}{2}, T^r (i) \in 2\mathbb{N}\\\\
3^{1+p_{r}} 2^{k-(r+1)} m+\frac{3T^{r} (i)+1}{2}, T^r (i) \in 2\mathbb{N}+1
\end{array}
\right\}\]
\[\;\;\;\;\;\;\;\;\;\;\;\;\;=3^{p_{r+1}}2^{k-(r+1)} m+T^{r+1} (i).\]
If step is \emph{\textbf{d.c.s.}}  then $ p_{r+1}=p_{r} $; conversely if step is \emph{\textbf{i.c.s.}} then $ p_{r+1}=1+p_{r} $. This shows that the relationship is true. Now, we can to do $ r = k $ obtaining the following expression:
\[T^{r} (2^{k} m+i) \arrowvert_{r=k}=3^{p_{r}} 2^{k-r} m+T^{r}(i)\arrowvert_{r=k}=3^{p_{k} } m+T^{k}(i)\]
\end{proof}
%.................................................................................................................................................................................................................................................................................................................................................................................................................................................................................................................................................................
\begin{lemma}Let $x_{0}$ be an arbitrary natural number, and let \textbf{o.c.p.} be a procedure initiated on $x_{0}$, then it is true the following relationship,
\begin{equation}
T^{n} (x_{0} )2^{\sum_{i=0}^{n} k_{i} }=3^{n}x_{0}+\sum_{r=1}^{n}(3^{n-r}2^{\sum_{i=0}^{r-1}k_{i}}),\quad k_{0}=0. 
\label{eqn20}
\end{equation}
\label{lemma8}
\end{lemma}
\begin{proof} If $n = 1$ we have $T(x_{0})2^{k_{1} }=3x_{0}+1$. The above relationship is in accord with (\ref{eqn10}). Let us assume in (\ref{eqn10}) one equation more, then 
\[3T^{n} (x_{0} )+1=2^{k_{n+1} } T^{n+1} (x_{0} )\]
, resolving with respect to $ T^ {n} (x _ {0}) $ and substituting in (\ref {eqn20}) will be,
\[\left(\frac{2^{k_{n+1}} T^{n+1} (x_{0} )-1}{3}\right)2^{\sum_{i=0}^{n}k_{i}}
=3^{n}x_{0}+\sum_{r=1}^{n}3^{n-r}2^{\sum_{i=0}^{r-1}k_{i} }.\]
Next, multiplying both members by $3$, and applying the distributive law in the left member is obtained,
\[2^{\sum_{i=0}^{n+1}k_{i}}  T^{n+1} (x_{0} ) -2^{\sum_{i=0}^{n}k_{i}} =3^{n+1}x_{0}+\sum_{r=1}^{n}3^{n+1-r}2^{\sum_{i=0}^{r-1}k_{i}}.\]
Adding in both members $2^{\sum_{i=0}^{n}k_{i}}$ we have the following,
\[2^{\sum_{i=0}^{n+1}k_{i}}  T^{n+1} (x_{0} )  =3^{n+1}x_{0}+\sum_{r=1}^{n}3^{n+1-r}2^{\sum_{i=0}^{r-1}k_{i}}+2^{\sum_{i=0}^{n}k_{i}}.\]
Incorporating $2^{\sum_{i=0}^{n}k_{i}}$ as part of the sum it is in the right member we obtain,
\[2^{\sum_{i=0}^{n+1}k_{i}}T^{n+1} (x_{0} )=3^{n+1}x_{0}+\sum_{r=1}^{n+1}3^{n+1-r}2^{\sum_{i=0}^{r-1}k_i}.\]
As it can be seen is obtained the expression (\ref{eqn20}) in which the only difference is the change of  $n$ by $n + 1$.
\end{proof}
%.................................................................................................................................................................................................................................................................................................................................................................................................................................................................................................................................................................
\begin{theorem}Let \textbf{g.c.p.}, initiated in an arbitrary number $ x_{0}\in \mathbb{N},  x_{0}>1$. Let $ A_{k}=P(T^{k} (x_{0})>T^{k-1} (x_{0})) $ and $B_{k}=P(T^{k} (x_{0})<T^{k-1} (x_{0})) $ be the probabilities of increment and decrement in a $k$ step, respectively. Then, if  $ A_{k}=B_{k},\forall k\in\mathbb{N} $,  then there exist $k$ such that $T^{k} (x_{0})<x_{0}$.
\label{theorem9}
\end{theorem}
\begin{proof} \textbf{First demonstration}: Let $ x_{0}$ be an arbitrary natural number and let \emph{\textbf{g.c.p.}} initiated in it. By virtue of lemma \ref{lemma8} we have the following expression for de odd numbers of \emph{\textbf{g.c.p.}},
\[T^{n} (x_{0} )=\frac{3^{n}x_{0}}{2^{\sum_{i=0}^{n}k_{i}}} +\frac{\sum_{r=1}^{n}3^{n-r}2^{\sum_{i=0}^{r-1}k_{i}} }{2^{\sum_{i=0}^{n}k_{i}}}, k_{0}=0.\]
By virtue of definition of \emph{\textbf{c.s.}} (see definition \ref{definition5}) the number $3$ it appear  exclusively in \emph{\textbf{i.c.s.}} $ (T(x_i)>x_i) $.  It mean that $n$ represent  increments only.\\On the other  hand, the number $2$ appear in all step; this mean that number $2$  which is in the denominator in right is as follow $\sum_{i=0}^{n}k_{i} =n+D_{n}-1$, being $D_{n}$ the amount of \emph{\textbf{d.c.s.}} $(T(x_i)<x_i)$; then we can to write the following:
\[T^n (x_0 )=x_0  \frac{3^n}{2^{n+D_{n}-1 }} +\frac{\sum_{r=1}^{n}3^{n-r}2^{\sum_{i=0}^{r-1}k_{_i}}}{2^{n+D_{n}-1 }} , k_0=0.\]
As $n+D_{n}-1 =(1+\frac{D_{n}-1}{n})n$ will be the following,
\[T^n (x_0 )=x_0  \frac{3^n}{2^{(1+\frac{D_{n}-1}{n})n }} +\frac{\sum_{r=1}^{n}3^{n-r}2^{\sum_{i=0}^{r-1}k_{_i}}}{2^{(1+\frac{D_{n}-1}{n})n }} , k_0=0.\]
However, if \emph{\textbf{g.c.p.}} is sufficient prolonged ($n$ is very large)  then we can to consider  $\frac{D_{n}-1}{n}=1$(In practice it is only an approximation $\frac{D_{n}-1}{n}\approx 1$ for to guarantee that $\frac{3^n}{2^{(1+\frac{D_{n}-1}{n})n }}\approx \frac{3^n}{4^{n}}<1$). The expression corresponding to $T^n (x_0 )$ can be now replaced by the fallowing,
\[T^{n} (x_0 )=x_0  \frac{3^n}{2^{2n}} +\frac{\sum_{r=1}^n3^{n-r} 2^{\sum_{i=0}^{r-1}k_i}}{2^{2n}},\]
and hence,
\begin{equation}
T^{n} (x_0 )=x_0 \left(\frac{3}{4}\right)^n+\left(\frac{3}{4}\right)^{n} \sum_{r=1}^{n}\frac{2^{\sum_{i=0}^{r-1}k_i}}{3^r}.
\label{eqn90}
\end{equation}
In the other hand, if we consider certain $m$ of additional steps we obtain that,
\[T^{n+m} (x_0 )=x_0\left(\frac{3}{4}\right)^{n+m}+\left(\frac{3}{4}\right)^{n+m} \sum_{r=1}^{n+m}\frac{2^{\sum_{i=0}^{r-1}k_i}}{3^r}.\]
Next, by decomposing of the sum, we have
\begin{equation}
T^{n+m} (x_0 )=x_0 \left(\frac{3}{4}\right)^{n+m}+\left(\frac{3}{4}\right)^{n+m} \sum_{r=1}^{n-1}\frac{2^{\sum_{i=0}^{r-1}k_i}}{3^r} +\left(\frac{3}{4}\right)^{n+m} \sum_{r=n}^{n+m}\frac{2^{\sum_{i=0}^{r-1}k_i}}{3^r}.
\label{eqn60}
\end{equation}
Let us consider now the potency $2^{\sum_{i=0}^{r-1}k_i}$. We know that $\sum_{i=0}^{r-1}k_i=r+D_{r}-2$ being $D_r$ the decrements until the step $r-1$ (See definition \emph{\textbf{o.c.p.}} and compare with definitions \emph{\textbf{i.c.s.}} and \emph{\textbf{d.c.s.}} respectively), therefore the expression will be,
\[2^{\sum_{i=0}^{r-1}k_i} = 2^{r+D_r-2}= 2^{r(1+\frac{D_r}{r})-2}.\]
It follows from this that (\ref{eqn60}) can to be write as,
\[T^{n+m} (x_0 )=x_0 \left(\frac{3}{4}\right)^{n+m}+\frac{1}{4}\left(\frac{3}{4}\right)^{n+m} \sum_{r=1}^{n-1}\frac{2^{r\left(1+\frac{D_r}{r}\right)}}{3^r} +\frac{1}{4}\left(\frac{3}{4}\right)^{n+m} \sum_{r=n}^{n+m}\frac{2^{r\left(1+\frac{D_r}{r}\right)}}{3^r}.\]
In accord with $r\geq n$ we assume in the last term of the above sum that $\frac{D_r}{r}=1 $ (In practice it is only an approximation $\frac{D_{n}}{n}\approx 1$ for to guarantee that $\frac{3^n}{2^{(1+\frac{D_{n}}{n})n }}\approx \frac{3^n}{4^{n}}<1$), and hence we have,
\begin{equation}   
\begin{array}{cc}
T^{n+m} (x_0 )=\\
\displaystyle{=x_0 \left(\frac{3}{4}\right)^{n+m}+\frac{1}{4}\left(\frac{3}{4}\right)^{n+m} \sum_{r=1}^{n-1}\frac{2^{r\left(1+\frac{D_r}{r}\right)}}{3^r} +\frac{1}{4}\left(\frac{3}{4}\right)^{n+m} \sum_{r=n}^{n+m} \left(\frac{4}{3}\right)^r}.
\end{array}
\label{eqn70}
\end{equation}
If $m$ goes to infinite we have,
\begin{equation}
^{lim}_{m\rightarrow\infty}\left(T^{n+m} (x_{0}) \right)=^{lim}_{m\rightarrow\infty}\frac{1}{4}\left(\frac{3}{4}\right)^{n+m} \sum_{r=n}^{n+m} \left(\frac{4}{3}\right)^r.
\label{eqn30}
\end{equation}
Using the following known formula,
\begin{equation}
\sum_{r=n}^{n+m} \left(\frac{4}{3}\right)^r=3\left(\frac{4}{3}\right)^{n}\left(\left(\frac{4}{3}\right)^{m+1}-1\right)
\label{eqn80}
\end{equation}
, we obtain the following,
\[^{lim}_{m\rightarrow\infty}\left(T^{n+m} (x_{0}) \right)=^{lim}_{m\rightarrow\infty}\frac{3}{4}\left(\frac{3}{4}\right)^{n+m} \left(\frac{4}{3}\right)^{n}\left(\left(\frac{4}{3}\right)^{m+1}-1\right)=\]
\[=^{lim}_{m\rightarrow\infty}\left(1-\left(\frac{3}{4}\right)^{m+1}\right),\]
and hence,
\[^{lim}_{m\rightarrow\infty}\left(T^{n+m} (x_0 )\right)=1.\]
However, this signify that exist a finite number $n+m$ such that
\[T^{n+m} (x_0 )< x_0. \]
\textbf{Second demonstration}: Beginning equation  (\ref{eqn70}) we take common factor $\left(\frac{3}{4}\right)^{n+m}$ in right  member, and hence,
\[T^{n+m} (x_0 )=\left(\frac{3}{4}\right)^{n+m}\left(x_0 +\frac{1}{4}\sum_{r=1}^{n-1}\frac{2^{r\left(1+\frac{D_r}{r}\right)}}{3^r}\right) +\frac{1}{4}\left(\frac{3}{4}\right)^{n+m} \sum_{r=n}^{n+m} \left(\frac{4}{3}\right)^r.\]
Let us use the inductive expression (\ref{eqn80}), and will be the following,
\[T^{n+m} (x_0 )=\left(\frac{3}{4}\right)^{n+m}\left(x_0 +\frac{1}{4}\sum_{r=1}^{n-1}\frac{2^{r\left(1+\frac{D_r}{r}\right)}}{3^r}\right) +1-\left(\frac{3}{4}\right)^{m+1}.\]
Next, by multiplying and dividing, the first sum of the right member by $\frac{4^{n-1}}{3^{n-1}}$ we obtain, 
\[T^{n+m} (x_0 )=\left(\frac{3}{4}\right)^{n+m}\frac{4^{n-1}}{3^{n-1}}\left(\frac{3^{n-1}}{4^{n-1}}x_0 +\frac{1}{4}\frac{3^{n-1}}{4^{n-1}}\sum_{r=1}^{n-1}\frac{2^{r\left(1+\frac{D_r}{r}\right)}}{3^r}\right) +1-\left(\frac{3}{4}\right)^{m+1}.\]
By comparing whit (\ref{eqn90}), we can to note that the sum in brackets is exactly $T^{n-1}(x_0)$, and hence,
\[T^{n+m} (x_0 )=\left(\frac{3}{4}\right)^{m+1}T^{n-1}(x_0) +1-\left(\frac{3}{4}\right)^{m+1}.\]
Let us assume now that $T^{n-1}(x_0)\leq T^{n+m} (x_0 )$ which strengthens the proof of convergence, is then,
\[T^{n+m} (x_0 )\leq\left(\frac{3}{4}\right)^{m+1}T^{n+m} (x_0 ) +1-\left(\frac{3}{4}\right)^{m+1}\]
, and hence we obtain,
\[T^{n+m} (x_0 )\leq1.\]
However, this signify that exist a finite number $n+m$ such that
\[T^{n+m} (x_0 )< x_0 .\]
The theorem is proofed by the second path.\end{proof}
%.................................................................................................................................................................................................................................................................................................................................................................................................................................................................................................................................................................

\begin{theorem}Let $\Gamma_{M}$ be a set of successive natural numbers since $1$ until $ 2^M$. Let \textbf{g.c.p.} initiated in an arbitrary number $x_0\in\Gamma_{M}$. Let $ A_{k}=P(T^{k} (x_{0})>T^{k-1} (x_{0})) $ and $B_{k}=P(T^{k} (x_{0})<T^{k-1} (x_{0})) $ be, respectively, the probabilities of increment and decrement in a \textbf{\emph{c.s.}} $k\leq M-1$, then the following is true:
\[A_k=B_k=\frac{1}{2},\forall x_i\in\Gamma_{M}.\]
\label{theorem10}
\end{theorem}
\begin{remark}Henceforth we imagine that Collatz procedure starts so only in a number, extracted randomly  from a finite  set  of successive numbers. In this case we have designated as $ \Gamma_ {M} $  a class of these finite sets that starts in the unit. If on this approach is demonstrated conjecture, it will be shown simultaneously in general, since every natural number can indeed be drawn at random from a finite set. You only have to enumerate some objects and perform with them the usual extraction experiment of such objects from an urn.
\label{remark1}
\end{remark}
\begin{proof}Let us suppose that the number $x_0$ is extracted from $\Gamma_{M}$, randomly, then the probability that $x_0$ it is even number, is equivalent to be odd, or $\frac{1}{2}$. By application of \textbf{\emph{c.s.}} on $x_0$, increment have a place when $\frac{(3 x_0+1)}{2}$, on the contrary, one decrement occurs when $\frac{x_0}{2}$, and we have: 
\[T(x_0 )=
\left\{\begin{array}[pos]{c}
\frac{x_0}{2}\ \ \ ,x_0\in2\mathbb{N}\\\\
\frac{3x_0+1}{2},x_0\in2\mathbb{N}+1
\end{array}\right\}.\]
Let us suppose now the elements of $\Gamma_{M}$ separated in accordance with equivalent class of the integers module $2^n$ such that $n\leq M-1$ (Each of these sets consists of $ 2 ^ {m-n} $ elements of the same equivalence class). Let us suppose in turn that each one of these class have the following properties:
\begin{enumerate}
\item Its elements generate during \emph{\textbf{g.c.p.}} initiated in them, a sequence whose step $n$ it is of type \emph{\textbf{d.c.s.}} ($T^n (x_0)<T^{n-1} (x_0)$ then we say that the set is of the type $C_i^\downarrow (n)$).
\item Its elements generate during \emph{\textbf{g.c.p.}} initiated in them, a sequence whose step $n$ it is of type \emph{\textbf{i.c.s.}} ($T^n (x_0)>T^{n-1} (x_0)$ then we say that the set is of the type $C_i^\uparrow (n)$).
\end{enumerate}
From lemma \ref{lemma7} it follows that the elements of the class $2^{n} m+i$ they arrive, in the step $n$ of \textbf{\emph{c.s.}}, at the number $3^{p_n } m+T^{n} (i)$ being $p_n\leq n$ the amount of increments occurred during $n$ Collatz steps; and hence it is the following: 
\[T^{n} (2^{n} m+i)=3^{p_{n} } m+T^{n} (i).\]
Depending on the parity of $m$ we have two cases:
\[T^{n} (2^{n} m+i)=T^{n}\left(\left\{
\begin{array}[pos]{c}
2^{n+1} r+i\ \ , if\ m=2r\\\\
2^{n+1} r+(2^{n}+i), if\ m=2r+1
\end{array}\right\}\right)\]
\[\;\;\;\;\;\;\;\;\;\;\;\;\;\;\;\;\;\;\;\;\;=\left\{
\begin{array}[pos]{c}
3^{p_n } 2r+T^{n} (i)\ ,if\ m=2r\\\\
3^{p_n } 2r+3^{p_n }+T^{n} (i),if\ m=2r+1
\end{array}\right\}.\]
It is to say, half of the numbers in the set $\Gamma_{M}$ they become in the step $n$ to a certain form (Superior form in the above square bracket), and the other half to the other form (Inferior form in the above square bracket).\\Next, in the step $n + 1$ we have the following: 
\begin{enumerate}
\item If $T^{n} (i)$ is an even number:
\begin{itemize}
\item  The number $3^{p_n} 2r+T^{n} (i)$ will be even and we have \emph{\textbf{d.c.s.}}:
\[T(3^{p_n } 2r+T^n (i))=3^{p_n } r+\frac{T^n (i)}{2}<3^{p_n } 2r+T^n (i).\]
\item In the other hand the number $3^{p_n } 2r+3^{p_n }+T^{n} (i)$ will be odd and we have \emph{\textbf{i.c.s.}}:
\[T(3^{p_n } 2r+3^{p_n }+T^{n} (i))=\frac{3^{p_{n+1}} 2r+3^{p_n+1}+3T^{n} (i)+1}{2}>3^{p_n } 2r+3^{p_n }+T^{n} (i).\]
\end{itemize}
\item If $T^{n} (i)$ is odd number:
\begin{itemize}
\item The number $3^{p_n } 2r+T^{n} (i)$ will be odd and we have \emph{\textbf{i.c.s.}}:
\[T(3^{p_n} 2r+T^{n} (i))=\frac{3^{p_{n+1}} 2r+3T^{n} (i)+1}{2}>3^{p_n} 2r+T^{n} (i).\]
\item Whoever $3^{p_n} 2r+3^{p_n}+T^{n} (i)$ will bi even and we have \emph{\textbf{d.c.s.}}:
\[T(3^{p_n } 2r+3^{p_n }+T^{n} (i))=\frac{3^{p_n } 2r+3^{p_n }+T^{n} (i)}{2}<3^{p_n } 2r+3^{p_n }+T^{n} (i).\]
\end{itemize}
\end{enumerate}
As $i$ is an arbitrarily selected number, such that $0 \leq i\leq 2^n$, the above have a place for all the equivalents classes of $\frac{\mathbb{Z}}{\mathbb{Z}_{2^n}}$, that is to say, in the step $n+1$ an \emph{\textbf{d.c.s.}} it is experienced by $\frac{2^{M}}{2}$ elements, while the other half have an \emph{\textbf{i.c.s.}}.\\ And therefore if $n\leq M-1$ (Which guarantees equal cardinality of the equivalence classes), in any step such that $n\leq M-1$ will be so many elements in classes $C_i^\downarrow (n)$ as in the classes $C_i^\uparrow (n)$. As it is true for $n = 1$, (As was seen it in the case $\mathbb{Z}_{2^1}=\mathbb{Z}_2$), then will be true for all numbers $n \leq M-1$. \\It is means that for all $x_0$ such that $x_0\in\Gamma_{M}$, randomly extracted, the probability of \textbf{\emph{d.c.s.}} it is identically to the probability of \textbf{\emph{i.c.s.}} in all \textbf{\emph{c.s.}} such that $n\leq M-1$. In other hand, everyone element of $\Gamma_{M}$ may be considered as randomly ``selected'' in an experiment with identical probabilities $\frac{1}{2^{M}}$ for any one of these.
\end{proof}\\
%.................................................................................................................................................................................................................................................................................................................................................................................................................................................................................................................................................................
Let us strengthen the understanding of the above theorem with concrete examples:\\
Example 1. 
\[\Gamma_{2}=\left\{1,2,3,4\right\}.\]
We have above that $\Gamma_{2}$ have the first element $1$ and the last is $2^{2}=4$ as the theorem say. In this case only one \textbf{\emph{c.s.}} is possible in which the theorem is fulfilled. In fact:
\[ T(1)=\frac{1*3+1}{2}=2,\]
\[T(2)=\frac{2}{2}=1,\]
\[T(3)=\frac{3*3+1}{2}=5,\]
\[T(4)=\frac{4}{2}=2.\]
As is anticipated by the theorem, the half of the elements of $\Gamma_{2}$,  have an \textbf{\emph{i.c.s.}} ($\left\{1,3 \right\}$), while the other half have \textbf{\emph{d.c.s.}} ($\left\{2,4 \right\}$).\\
Example 2.
\[\Gamma_{3}=\left\{1,2,3,4,5,6,7,8\right\}.\]
First step:
\[ T(1)=\frac{1*3+1}{2}=2,\]
\[ T(2)=\frac{2}{2}=1,\]
\[ T(3)=\frac{3*3+1}{2}=5,\]
\[ T(4)=\frac{4}{2}=2,\]
\[ T(5)=\frac{5*3+1}{2}=8,\]
\[ T(6)=\frac{6}{2}=3,\]
\[ T(7)=\frac{7*3+1}{2}=11,\]
\[ T(8)=\frac{8}{2}=4.\]
In this case the numbers $\left\{2,4,6,8\right\}$ have a \textbf{\emph{d.c.s.}} while $\left\{1,3,5,7\right\}$ have a \textbf{\emph{i.c.s.}} as expected.\\
Step two:
\[ T^{2}(1)=T(2)=\frac{2}{2}=1,\]
\[ T^{2}(2)=T(1)=\frac{3*1+1}{2}=2,\]
\[ T^{2}(3)=T(5)=\frac{3*5+1}{2}=8,\]
\[ T^{2}(4)=T(2)=\frac{2}{2}=1,\]
\[ T^{2}(5)=T(8)=\frac{8}{2}=4,\]
\[ T^{2}(6)=T(3)=\frac{3*3+1}{2}=5,\]
\[ T^{2}(7)=T(11)=\frac{3*11+1}{2}=17,\]
\[ T^{2}(8)=T(4)=\frac{4}{2}=2.\]
As the theorem it states the numbers $\left\{2,8,4\right\}$ that were obtained in the first step, have \textbf{\emph{d.c.s.}} in the step two (We have in count that the number $2$ appear twice in the first step, it is to say $1\longrightarrow 2$ and $4\longrightarrow 2$), while $\left\{1,5,3,11\right\}$ that were obtained in the first step, have \textbf{\emph{i.c.s.}} in the step two. Theorem is confirmed again.\\ 
The theorem does not hold for another step in this case because the most largest of the numbers of $\Gamma_{3}$ is the third power of $2$, while the theorem is true for $M-1$ being $2^{M}$ the greater number of $\Gamma_{M}$.\\The reader can, following this idea, check the theorem in more extensive cases.
%.................................................................................................................................................................................................................................................................................................................................................................................................................................................................................................................................................................
\begin{corollary}The Collatz conjecture is true.
\label{corollary1}
\end{corollary}
\begin{proof}First, we perform some important statistic considerations:\\ We have referred many times to the consideration of $ n $ large enough. With respect this we have to perform the following experiment:\\ Let us generate random numbers  $x_{n+D_{n}}$ in the set $\left\{0,1\right\}\subset \mathbb{N}$ (This is a proses analogous  from statistic point of view to a sequence  of \textbf{\emph{c.s.}} of increment (numbers 1) and  decrement (numbers 0) in a \textbf{\emph{p.c.}}). Have place that, the relationship $\xi_{n+D_{n}}=\frac{D_{n}}{n}$  where quantity  of zeros (Let us denote as $D_{n}$ by analogy with the decrements) and unities (Let us denote as $n$ by analogy with the increments) that are generated it is sufficiently close to the unit for a sequence $x_{0},x_{1},\ldots,x_{n+D_{n}}$ with  $n+D_{n}$ not as big. Let us suppose that $n+D_{n}=100$, and we consider $14$ values of $\xi_{100}$. Let us denote as $s_{100}$ the standard deviation of the each sample. Let us present the results:\\ 
\[\begin{array}{ccccc}
sample&\xi_{100}&1+\xi_{100}&s_{100}&2^{1+\xi_{100}}\\
1&0,7241&1,7241&0,4960&3,3038\\
2&1,0833&2,0833&0,5021	&4,2379\\
3&1,1277&2,1277&0,5016&4,3701\\
4&1,0833&2,0833&0,5021&4,2379\\
5&1,3256&2,3256&0,4976&5,0127	\\
6&0,7857	&1,7857&0,4989&3,4479\\
7&1,0833&2,0833&0,5021&4,2379	\\
8&1,0833	&2,0833&0,5021&4,2379\\
9&1,8571&2,8571&0,4794&7,2458\\
10&0,8519&1,8519&0,5009&3,6096\\
11&1,0408&2,0408&0,5024&4,1148\\
12&0,9231&1,9231&0,5021&3,7923\\
13&0,9608&1,9608&0,5024&3,8927\\
14&1,1739&2,1739&0,5009&4,5125
 \end{array}\]
Having in mine the above data we can to perform the following table (We suppose an normal distribution for !'$amount\ of\ samples\gg 14$! it which is reasonable):
 \[\begin{array}{cc}
confidence (\%) &interval \ for\  \mu=E(1+\xi_{100}) \\
95	&3,8953 \leq  \mu \leq	4,9174 \\
98	&3,7794	\leq  \mu  \leq	5,0333	\\
99	&\ \ \ \ 3,6938	\leq  \mu \leq	5,1189\ \ \ \  ,	
\end{array}\]
and hence,
\[\begin{array}{cc}
confidence (\%) &interval \ for\  \chi=2^\mu \\
95	&3,8953 \leq  \chi \leq	4,9174 \\
98	&3,7794	\leq  \chi  \leq	5,0333	\\
99	&3,6938	\leq  \chi \leq	5,1189	
\end{array}\]
For generating random numbers in each sample is used the corresponding function of the electronic tabulator Microsoft Excel 2010. For determine the reasons between the zeros and unities, and the confidence intervals of exponential function, statistical SPSS system version 15.0 is used.\\The last table justifies our change of $\frac{3}{2^{1+\frac{D_{n}}{n}}}$ by $\frac{3}{4}$ in the proof of theorem \ref{theorem9} without having gone to infinity.\\ The above it is in accord with the paper of David Applegate and Jeffrey C. Lagarias\cite{7}$(2003)$  under which a total stopping time could be $\sigma_{\infty}(x) > 6.14316 \log(x)$, then a number $x$ that has $\sigma_{\infty}(x)=100$ it implies that $1.17371*10^7>x$; it is to say, the total stopping time is very low for really large numbers. In this case we have $1.17371*10^7<2^{101}$ that it belong to the set $\Gamma_{101}$ and hence the theorem \ref{theorem10} guarantees an amount of $100$ steps with probability $\frac{1}{2}$ for increment and decrement respectively. But in accord with the second demonstration of theorem \ref{theorem9}, an amount of $100$ steps are sufficient to ensure convergence to $1$.
\begin{remark}It is very important to say that, when we say that $100$ steps are sufficient to
ensure the convergence, that does not mean in any way that the convergence actually it is produced in an amount of $100$ steps in the general case; precisely why, in theorem \ref{theorem9} the sum $n+m$ appears and not just $n$;  meanwhile $n$ guarantees the convergent form of the law that generating the elements of the succession, while $m$ leads that form the effective convergence.
\label{remark2}
\end{remark}
From this it follow that we can to ensure that any set $\Gamma_{M}, M\geq 101$ (Provided that $ M $ is such that in principle allows the extraction of random numbers) guaranteed by theorem \ref{theorem10}, an amount of $100$ \textbf{\emph{c.s.}}, with probability $\frac{1}{2}$ for \textbf{\emph{d.c.s.}} and \textbf{\emph{i.c.s.}} respectively in each one of steps, for all its elements, which is more than enough to fulfilling the theorem \ref{theorem9} for all this elements.\\ But if all numbers are \textbf{\emph{c.n.}} until certain $n\in\Gamma_{M}, n\leq 2^{M}-1$, then $n+1$ is \textbf{\emph{c.n.}} because for certain step $k$ it will be $T^{k}(n+1)\leq n$.\\ However, in order to proof the Collatz conjecture is required, to consider the following limit $M\longrightarrow \infty$ but this operation makes impossible the random extraction in $\Gamma_{\infty}\equiv \mathbb{N}$ in principle, which leaves invalid to theorem \ref{theorem10}. We resolve this problem of the following way:\\ Let us divide the set $\mathbb{N}$ in classes with the same amount of elements, each one as,
\[\left[1,2^{M}\right],\left[2^{M}+1,2^{M+1}\right],\left[2^{M+1}+1,2^{M+1}+2^{M}\right],...,\left[p,q\right],\left[p+2^{M},q+2^{M}\right],....\]
Led us consider $M\geq 400$ such that it allows, in principle, the random extracting of elements from each of these classes. The first of these classes is $\Gamma_{M}$ and the others are isomorphic with this first class; then theorem \ref{theorem10} is true.\\ Now, if until class $Cl_{n}$ all the numbers are \textbf{\emph{c.n.}} then the first number $1_{n+1}$  of the class $Cl_{n+1}$ it will be \textbf{\emph{c.n.}} too, because by virtue of size of $Cl_{n+1}$ theorem \ref{theorem9} will be satisfied and will exist certain step $k$ such that $T^{k}(1_{n+1})\in Cl_{n}$. Like all other elements of  $Cl_{n+1}$ satisfy the theorem \ref{theorem9}, then we can repeat the reasoning and conclude that $Cl_{n+1}$ is entirely of Collatz. By Induction all classes considered are Collatz classes, and thus all elements of $\mathbb {N}$.\end{proof}
%.................................................................................................................................................................................................................................................................................................................................................................................................................................................................................................................................................................
\section{Discussion and complementary themes}
\subsection{Necessary probabilistic character of the problem \\ $ 3n + 1 $}The analysis of the probability is essential to the solution of the conjecture $3n + 1$. In fact, you can not reach solution without considering the probability. In order to show it, you can see that if a certain number $n$ of expressions, given in (\ref {eqn10}), are multiplied member by member, the following should be,
\[ \prod^{n}_{i=0} (3 x_{i}+1)  = 2^{\sum^{n}_{i=0}{k_{i}}} \prod^{n+1}_{i=1}x_{i}.\]
By simple transformations we have,
\[ \frac{3 x_{0}+1}{x_{n+1}}\prod^{n}_{i=1} (3 +\frac{1}{x_{i}})=2^{\sum^{n}_{i=0}{k_{i}}}.\]
Now, if the substitution $x_{i}\rightarrow 1$ is performed, will be
\[ \frac{3 x_{0}+1}{x_{n+1}}(4)^{n} \geq 2^{\sum^{n}_{i=0}{k_{i}}}.\]
And solving it with respect to $ x_{n + 1}$, we have,
\[ (3 x_{0}+1)\frac{(4)^{n}}{2^{\sum^{n}_{i=0}{k_{i}}}} \geq x_{n+1}.\]
Next, by putting $\sum^{n}_{i=0}{k_{i}}=(n+1)\bar{k_{n}}$ is the following,
\[ \frac{3 x_{0}+1}{4}\left(\frac{4}{2^{\bar{k_{n}}}}\right)^{n+1} \geq x_{n+1}.\]
From the above it is clear, as the fulfillment of conjecture depends on a statistical variable, i.e. the behavior of the mean value $ \bar{k} $ for large values of $n$. As we saw above $\bar{k_{n}}=1+\frac{D_{n}}{n}$, where $D_{n}$ are decreases. Hence, if for large values of $n$, can be considered $D_{n}\geq n$ conjecture will be true. But this consideration cannot be performed clearly categorically; it is possible only with a probability approach. Such is the path we have followed here considering, based on the above, that is the most promising and perhaps the only one way for the final demonstration. Many studies have addressed the issue from this point of view as in the case of  D. Gluck and B. Taylor\cite{4}$(2002)$, and Y. Sinai\cite{5}$(2003)$ and everything is worth noting the work of Alain Slakmon and Luc Macot\cite{6}$(2006)$. These last authors consider even equal probabilities to $1/2 $ which was initially considered by them and has been shown here for us.
%.................................................................................................................................................................................................................................................................................................................................................................................................................................................................................................................................................................
\subsection{About the general problem $an+b$}
Theorem \ref{theorem9} is not valid for the general case $an+b$ where $a\geq 5$, $b\geq 1$ are odd numbers and $a\geq b+2$. In this case the expression is (This may prove extremely analog way) as follow,
\[T^{n} (x_{0} )=\frac{a^{n}x_{0}}{2^{\sum_{i=0}^{n} k_{i} }}+\frac{b}{2^{\sum_{i=0}^{n}k_{i} }}\sum_{r=1}^{n}(a^{n-r}2^{\sum_{i=0}^{r-1}k_{i}}),\quad k_{0}=0.\]
All analog demonstrative process then would arrive at the expression (\ref{eqn30}) in general form
\begin{equation}
T^{n+m} (x_{0}) =x_{0} \left(\frac{a}{4}\right)^{n+m}+b\left(\frac{a}{4}\right)^{n+m} \sum_{r=1}^{n-1}\left(\frac{2}{a}\right)^r +b\left(\frac{a}{4}\right)^{n+m} \sum_{r=n}^{n+m} \left(\frac{4}{a}\right)^r.
\label{eqn100}
\end{equation}
But as $a\geq 4$, the convergence is impossible to the right. However this result is entirely predictable, since, in general, these procedures lead to cycles that has $ 2 $ or more odd numbers, which is not convergence, whereas the case $ 3n + 1 $ is up to a cycle of one odd element, that is to say \emph{\textbf{o.c.p.}} converge to the unity.\\In the case $5n+1$ we have, for example, the following cycles $[3,1]$, $[13,33,83]$ y $[17,43,27]$. In the other hand the case initiated in $7$ appears to go rapidly to infinity:
\[7
,9
,23
,29
,73
,183
,229
,573
,1433
,3583
,4479
,5599
,6999
,8749
,21873
,54683,\]
\[,34177
,85443
,26701
,66753
,166883
,52151
,65189
,162973
,407433
,1018583,\]
\[,1273229
,3183073
,7957683
,310847
,388559
,485699
,151781
,379453
,948633,\]
\[,2371583,2964479,3705599,4631999,5789999,\ldots\]
In the case $7n+1$ cycle $\left[1,1,...\right]$ occurs, so you have the procedure completed $\left[5,9,1\right]$, but the case of number seven is increasing dramatically:
\[7
,25
,11
,39
,137
,15
,53
,93
,163
,571
,1999
,6997
,12245
,21429
,37501,\]
\[,65627
,229695
,803933
,1406883
,4924091
,17234319
,60320117
,105560205,\]
\[,184730359
,646556257
,565736725
,990039269\]
We can say that the procedures \emph{\textbf{o.c.p.}} of the form $an+b$ ($a\geq b+2$, with $a$ y $b$ odd) do not generate convergent sequences: It is sufficient to show that in all these cases we have, 
\[P(T^{k} (x_{0})>T^{k-1} (x_{0})) \geq P(T^{k} (x_{0})<T^{k-1} (x_{0}))\]
, where $P$ is the probability and, of course, $T(x)$ would be defined as,
\[T^{0} (x_{0} )= x_{0}\in2\mathbb{N}+1,\]
\[T^{i}(x_{0} )= \frac{aT^{i-1}(x _{0})+b}{2^{k_{i} }},\quad k_{i}\in\mathbb{N},\quad T^{i} (x_{0})\in2\mathbb{N}+1.\]
It can be shown then, that the lemma \ref{lemma7} and theorem \ref{theorem10} (Whose accuracy depends on the lemma \ref{lemma7}) are perfectly valid for the general case: The lemma \ref{lemma7} would enunciated as follows,``\emph{Let $ 2^{k} m+i (0 \leq i< 2^k) $ be a equivalent class of additive group $ \mathbb{Z}_{2^k} $ module $ 2^k $  the relationship $T^k (2^{k} m+i)=a^{p_{k}} m+T^{k} (i) $ is satisfied, being $p_{k} $ the amount of \textbf{i.c.s.} that occur in $k$ steps.}'' (Note that it has changed only $ 3 $, by $ a $), and it would prove so very analogously. Meanwhile theorem \ref{theorem10} is would enunciate identically, and it would proved similarly, considering the change in the lemma \ref{lemma7} that not change anything in the reasoning or the result.\\Therefore, sequences generated by these procedures (By means of (\ref{eqn100})) They do not converge in general, so you can make sure that there is nothing analogous to the \emph{Collatz Conjecture}  for more general problems that $3n + 1$.
%.................................................................................................................................................................................................................................................................................................................................................................................................................................................................................................................................................................
\section{Conclusions}
We believe that the fundamental objective of our work has been achieved, that is, from our point of view the Collatz Conjecture is demonstrated. We know that many algebraists have no confidence in the results obtained by using probabilistic concepts; for our part this approach is scientific and rigorous long as your application is fully justified. Precisely, one of our results is the demonstration of the necessity of probabilistic approach to this problem, a goal that was also achieved.
Finally, we have made it clear that, in general, similar to the Collatz conjecture statement for all cases of expressions ax + b cannot be established.\\\\
%.................................................................................................................................................................................................................................................................................................................................................................................................................................................................................................................................................................
\textbf{\emph{Acknowledgment}}:\\ 
\emph{ Do not want to miss this opportunity to warmly thank the help of Professor MSc. Eligio Barreto Fiu, for their support in the topic of statistic.}
%.................................................................................................................................................................................................................................................................................................................................................................................................................................................................................................................................................................

\end{document}